\documentclass[a4paper, 12pt]{amsart}

\usepackage{amssymb,amscd}
\usepackage{color} 

\def\opn#1#2{\def#1{\operatorname{#2}}} 
\opn\Ker{Ker}
\opn\Im{Im}
\opn\rank{rank}
\opn\ini{in}
\opn\PF{PF}
\opn\Fr{F}
\opn\Ap{Ap}
\opn\K{K}
\opn\Soc{Soc}
\opn\type{type}
\opn\supp{supp}
\opn\mult{mult}
\opn\cone{cone}
\opn\relint{relint}
\opn\Ann{Ann}
\def\AGG{almost Gorenstein graded }
\newcommand{\Z}{\mathbb Z}
\newcommand{\N}{\mathbb N}
\newcommand{\Q}{\mathbb Q}
\def\e{\boldsymbol{e}}
\def\mm{\mathfrak m}
\def\qq{\mathfrak q}
\def\aa{\mathfrak a}
\def\bb{\mathfrak b}
\newcommand{\fd}{k}
\newcommand{\wdef}[1]{\textit{#1}}

\newtheorem{Theorem}{Theorem}[section]
\newtheorem{thm}[Theorem]{Theorem}
\newtheorem{prop}[Theorem]{Proposition}
\newtheorem{lem}[Theorem]{Lemma}
\newtheorem{cor}[Theorem]{Corollary}
\newtheorem*{claim}{Claim} 

\theoremstyle{definition}
\newtheorem{rem}[Theorem]{Remark}
\newtheorem{ex}[Theorem]{Example}
\newtheorem{defn}[Theorem]{Definition}
\newtheorem{note}[Theorem]{Notation}

\begin{document}

\title{Almost Gorenstein simplicial semigroup rings}

\author{Kazufumi Eto}
\address{Department of Mathematics, Nippon Institute of Technology, 
Miyashiro, Saitama, Japan 345-8501
}
\email{etou@nit.ac.jp}
\author{Naoyuki Matsuoka}
\address{Department of Mathematics, School of Science and Technology, Meiji University,
1-1-1 Higashi-mita, Tama-ku, Kawasaki, Japan 214-8571
}
\email{naomatsu@meiji.ac.jp}
\author{Takahiro Numata}
\address{Department of Mathematics, Kanagawa Institute of Technology, 
1030 Shimo-ogino, Atsugi, Kanagawa, Japan 243-0292
}
\email{takahiro.n0913@gmail.com}
\author{Kei-ichi Watanabe}
\address{Department of Mathematics, College of Humanities and Sciences,
Nihon University, Setagaya-ku, Tokyo, 156-8550, Japan and Organization
for the Strategic Coordination of Research and Intellectual Properties, Meiji
University}
\email{watnbkei@gmail.com}

\date{\today}
\thanks{Kei-ichi Watanabe was partially supported by JSPS Grant-in-Aid for Scientific Research (C) 23K03040}
%


\begin{abstract}
We give a criterion for almost Gorenstein property 
for semigroup rings associated with simplicial semigroups. 
We extend Nari's theorem for almost symmetric numerical semigroups 
to simplicial semigroups with higher rank. 
By this criterion, we determine $2$-dimensional normal semigroup rings
which have ``Ulrich elements'' defined in \cite{HJS}. 
\end{abstract}


\maketitle

\markboth{K. Eto, N. Matsuoka, T. Numata and K. -i. Watanabe}%
{Almost Gorenstein
simplicial semigroup rings}

\section{Introduction}

The concept of almost Gorenstein rings, a notable extension of Gorenstein rings in the field of commutative algebra, has evolved significantly since its inception. 
Initially introduced in 1997 by Barucci and Fr\"{o}berg \cite{BF}, the definition was specifically designed for one-dimensional analytically unramified local rings, based on a detailed analysis of almost symmetric numerical semigroups. 
In 2013, Goto, Matsuoka, and Phuong \cite{GMP} proposed a broader definition for almost Gorenstein properties in arbitrary one-dimensional Cohen-Macaulay local rings, equivalent to the definition by Barucci-Fr\"{o}berg \cite{BF} under the assumption that the base ring is analytically unramified.
The framework was further extended in 2015 by Goto, Takahashi, and Taniguchi \cite{GTT} to apply to Cohen-Macaulay local/graded rings of general dimensions. 
 In contrast, almost Gorenstein properties are defined by the embedding of the base ring $R$ within the canonical module $\K_R$ of $R$, where the quotient module $\K_R / fR$ becomes controllable known as an Ulrich module for some $f \in \K_R$.
In the case of graded rings, this definition is adapted to involve the graded canonical module. 
The precise definitions of the concepts will be given later in this paper.

In what follows, let $k$ be a field. Considering the foundation of almost Gorenstein rings lies in the theory of numerical semigroup rings as established by Barucci and Fr\"oberg \cite{BF}, it becomes a natural question to ask when higher-dimensional semigroup rings become almost Gorenstein. 
Herzog, Jafari, and Stamate \cite{HJS} have provided sufficient conditions for the almost Gorenstein property in semigroup rings associated with normal simplicial affine semigroups. 
More precisely, they have defined Ulrich elements for a semigroup $H$, and presented a necessary and sufficient condition for the existence of these elements. 
The existence of an Ulrich element $u$ in $H$ leads to the fact that $R=\fd[H]$ is almost Gorenstein, because $K_R / t^u R$ is an Ulrich module.

In this paper, we generalize their result to general (not necessarily normal) 
simplicial affine semigroup rings.
Furthermore, we will give a criterion for almost Gorenstein property via symmetry of the generators 
of the canonical module,  which is a natural generalization of Nari's criterion \cite{Nari13}, 
for almost Gorenstein property of numerical semigroup rings. 
Using this criterion, we completely determine the 2-dimensional normal semigroups which have an 
Ulrich element in Theorem 7.2.

As an application of our results, we are able to determine when Ulrich elements exist in the case of two-dimensional normal affine semigroups. 
Namely, if $\fd[H]$ is the 
invariant subring of the cyclic group generated by 
$\begin{pmatrix} \xi & 0 \\ 0 & \xi^{m} \end{pmatrix}$, where $\xi$ is a primitive $n$-the root of 
unity and $m(<n)$ is a positive integer relatively prime to $n$, then $\fd[x,y]^G$ has an Ulrich element 
if and only if either $m=1$ or the Hirzebruch-Jung continued fraction expansion 
of $\dfrac{n}{m}$ is of the form $[[q+1, \underbrace{2, \dots, 2}_{c-1}, p+1]]$.


In the following, we will outline the structure and key contributions of this paper. 

In Section 2, we begin by revisiting the definitions of simplicial semigroups and associated concepts. This includes establishing the Ap\'{e}ry set and socle elements, along with a review of the definition of the type of a semigroup. 
A crucial concept for this paper, the orthogonal semigroup, is defined also in Section 2. 
We demonstrate that any simplicial semigroup can be transformed into an orthogonal semigroup through equivalent deformations. 
Additionally, we describe the normality of orthogonal semigroups in terms of the Ap\'{e}ry set.

Section 3 is dedicated to recalling the definition of the Hilbert series of graded modules over graded rings and reminding the concept of multiplicity. 
We characterize the Cohen-Macaulay property of the semigroup ring $\fd[H]$ of a simplicial semigroup $H$ using the Hilbert series and multiplicity. 
Moreover, for the case where $R = \fd[H]$ is Cohen-Macaulay, we provide an evaluation of the multiplicity of the residue module $\K_R / fR$ of the graded canonical module $\K_R$, where $f \in \K_R$ is a homogeneous element.

In Section 4, we review the definition of almost Gorenstein properties for Cohen-Macaulay graded rings with a few comments on almost Gorenstein local rings and introduce the concept of AG semigroups. 
The section explores the contrast and necessary and sufficient conditions between the AG property of a semigroup and the almost Gorenstein property of its semigroup ring. 
We further characterize the AG property of a semigroup using socle elements, extending the results of Nari. This section is enriched with numerous concrete examples.

\section{Simplicial semigroups}

We first introduce notations and definitions for semigroups.
In this paper, we use the term `semigroup' to refer specifically to what is traditionally known as a `monoid'. This means that our semigroups are not only closed under addition but also contain the identity element $0$. 

Let $\Z$ denote the ring of integers and $\N$ the set of nonnegative integers. 
Consider $H$ as a commutative semigroup in this sense. Throughout this paper, we always assume that $H$ is cancellative. Hence the Grothendieck group $G(H)$ generated by $H$ exists and $G(H)$ is abelian. 
Note that $\N$ acts on $H$ in the usual way, typically by addition. The \wdef{rank} of $H$, denoted by $\rank H$, is defined as the rank of $G(H)$ as a $\Z$-module. We define a partial order $\leq_H$ on $H$ as follows:
\[
h\leq_H h'
\quad\Longleftrightarrow\quad h'-h\in H,
\qquad
\text{for } h, h'\in G(H).
\]

\begin{defn}\label{relint}
Let $H$ be a semigroup.
If $H$ is isomorphic to a subsemigroup in $\Z^d$
for some $d>0$,
then it is called an \wdef{affine} semigroup.
We always assume that 
an affine semigroup is \wdef{positive},
equivalently
it is isomorphic to a subsemigroup in $\N^d\subset\Z^d$.
\end{defn}

\begin{defn}
Let $H$ be an affine semigroup of rank $d > 0$. 
We say that $H$ is \wdef{simplicial} if there exist $b_1, b_2, \dots, b_d \in H$ and $n > 0$ such that $nH \subset \sum_{i=1}^d \N b_i$, 
where $nH = \{nh \,|\, h \in H\}$. 
When this is the case, $b_1, b_2, \dots, b_d$ is said to be the \wdef{extreme rays} of $H$. 
\end{defn}

When $H$ is a simplicial semigroup of rank $d$, 
we may assume $H\subset\Z^d$ 
due to the isomorphism $G(H)\cong\Z^d$ as groups.
Note that an affine semigroup need not be simplicial in general.

\begin{note} Let $H$ be an affine semigroup.
\begin{enumerate}
\item 
We say that  $(H, E)$ is a simplicial semigroup if $H$ is a simplicial 
semigroup and $E$ is the set of all extreme rays of $H$. 
Notice that 
$\N E = \sum_{h \in E} \N h$ is a subsemigroup of $H$ 
and the semigroup ring $\fd[\N E]$ is a polynomial ring in $d$ variables.
\item 
We denote $H$ by a matrix whose column vectors generate $H$.
For example, 
$H=\begin{pmatrix}2 & 0 \\ 0 & 3\end{pmatrix}$
means that $H$ is a semigroup in $\Z^2$ 
generated by $\binom20$ and $\binom03$.
\item 
For $h\in\Z^d$,
we denote the $i$-th entry of $h$ by $\sigma_i(h)$
for each $i$.
\end{enumerate}
\end{note}

\begin{ex}
Let 
\[
H=
\begin{pmatrix}
1 & 1 & 0 & 0 \\
1 & 0 & 1 & 0 \\
0 & 1 & 0 & 1 \\
0 & 0 & 1 & 1
\end{pmatrix}.
\]
Then $H$ is an affine semigroup but not simplicial,
since $H$ possesses four extreme rays
and the rank of $H$ is three.
\end{ex}

We recall the definition of the Ap\'ery set.

\begin{defn}\label{type} Let $H$ be an affine semigroup and
$E$ an arbitrary subset of $H$.
\begin{enumerate}
\item We define the \wdef{Ap\'ery set} of $E$ in $H$ as follow:
\[
\Ap(H, E)=\{ h\in H \,|\, h-b\notin H\text{ for each }b\in E\}.
\]
\item Let $\Soc(H,E)$ denote the set of all maximal elements in $\Ap(H,E)$ 
with respect to the order $\leq_H$.
When $(H, E)$ is a simplicial semigroup,
$\Ap(H, E)$ is finite and the cardinality of $\Soc(H, E)$ is called the \wdef{type} of $H$, denoted by $\type H$
(\cite[Definition 3.1]{JY}).
\item Note that, for every $h\in G(H)$, there exists $v\in\Ap(H, E)$ such that $h-v\in\Z E$.
In fact, if $h\in H$, this is clear by definition.
On the other hand, if $h\notin H$, we can choose $h_1, h_2\in H$ with $h=h_1-h_2$, and there exists $c>0$ such that $ch_2\in\N E$. 
Consequently, we have $h+ch_2=h_1+(c-1)h_2\in H$, establishing the existence of $v \in \Ap(H,E)$ as explained above.
\item 
For $h\in H$ and $v\in\Ap(H, E)$,
we denote $v = h^{\vee}$ if  $h + v\in\Z E$.
Note that $h+ h^{\vee}\in\N E$ and $0^{\vee}=0$. 
It might not be unique in general.
\end{enumerate}
\end{defn}

\begin{ex}
Let 
\[
H=
\begin{pmatrix}
2 & 0 & 0 & 1 & 1 \\
0 & 2 & 0 & 1 & 0 \\
0 & 0 & 2 & 0 & 1 
\end{pmatrix}
\quad\text{and}\quad
E=
\begin{pmatrix}
2 & 0 & 0 \\
0 & 2 & 0 \\
0 & 0 & 2 
\end{pmatrix},
\]
i.e., 
$H$ is a semigroup
generated by the column vectors of the matrix $\begin{pmatrix}
2 & 0 & 0 & 1 & 1 \\
0 & 2 & 0 & 1 & 0 \\
0 & 0 & 2 & 0 & 1 
\end{pmatrix}
$
and $E$ is a set
consisting of the column vectors of $\begin{pmatrix}
2 & 0 & 0 \\
0 & 2 & 0 \\
0 & 0 & 2 
\end{pmatrix}$.
Then
\[
\Ap(H, E)
=\left\{
\begin{pmatrix}0\\0\\0\end{pmatrix},
\begin{pmatrix}1\\1\\0\end{pmatrix},
\begin{pmatrix}1\\0\\1\end{pmatrix},
\begin{pmatrix}2\\1\\1\end{pmatrix}
\right\},~
\Soc(H, E)
=\left\{
\begin{pmatrix}2\\1\\1\end{pmatrix}
\right\}.
\]
\end{ex}

For $d > 0$ and each $1 \le i \le d$, 
$\e_i \in \Z^d$ denotes the column vector 
where the $i$-th component is $1$ and all other components are $0$.

\begin{defn}
Let $(H, E)$ be a simplicial semigroup of rank $d$ and $m > 0$ a positive integer.
If the following two conditions are satisfied,
we say that $H$ is an \wdef{orthogonal} semigroup of order $m$:
\begin{enumerate}
\item
$H\subset\Z^d=\bigoplus_{i=1}^d\Z\e_i$ and
$E=\{m\e_i\}_{i=1, \dots, d}$,
\item
if $0<m'<m$, then $m'\e_i\notin H$ for each $1 \le i \le d$.
\end{enumerate}
\end{defn}

\begin{lem}\label{orth_semi}
If $H$ is a simplicial semigroup,
then there exists an orthogonal semigroup $H'$
isomorphic to $H$ as a semigroup.
\end{lem}

\begin{proof}
Let $d$ be the rank of $H$ and 
$E$ the set of extreme rays of $H$.
Since $G(H)$ is isomorphic to $\Z^d$,
we may assume that $H$ is contained in $\Z^d$ and 
we can write $E=\{b_1,b_2,\ldots, b_d\}$
where, for each $1 \le i \le d$, $b_i\in\Z^d$ is minimal,
i.e. $m'b_i\notin H$ for $m'\in\Q$ with $0<m'<1$.
Since $\rank_{\Z}\sum_{i=1}^d\Z b_i=d$,
we may assume $\det(\sigma_i(b_j))_{0\leq i,j\leq d}>0$.
Consider the adjoint of the matrix
$(\sigma_i(b_j))_{0\leq i,j\leq d}$,
which defines an injection $\iota :\Z^d\to\Z^d$.
Then $H'=\iota(H)$ is orthogonal.
\end{proof}

\begin{ex}[{\cite[Example 2.2]{JY}}]
Let
\[
H=\begin{pmatrix}
5 & 1 & 8 & 2 & 2 \\
3 & 5 & 3 & 1 & 2 \\
1 & 2 & 5 & 1 & 1
\end{pmatrix}. 
\]
Considering the adjoint matrix of the submatrix of $H$ formed by the first three columns, we get
\[
\text{adj } 
\begin{pmatrix}
5 & 1 & 8  \\
3 & 5 & 3  \\
1 & 2 & 5 
\end{pmatrix}
=
\begin{pmatrix}
19 & 11 & -37 \\
-12 & 17 & 9 \\
1 & -9 & 22
\end{pmatrix}
\]
and hence
\[
H\cong
\begin{pmatrix}
19 & 11 & -37 \\
-12 & 17 & 9 \\
1 & -9 & 22
\end{pmatrix}
\begin{pmatrix}
5 & 1 & 8 & 2 & 2 \\
3 & 5 & 3 & 1 & 2 \\
1 & 2 & 5 & 1 & 1
\end{pmatrix}
=\begin{pmatrix}
91 & 0 & 0 & 12 & 23 \\
0 & 91 & 0 & 2 & 19 \\
0 & 0 & 91 & 15 & 6
\end{pmatrix}.
\]
We then have the semigroup $H' = \begin{pmatrix}
91 & 0 & 0 & 12 & 23 \\
0 & 91 & 0 & 2 & 19 \\
0 & 0 & 91 & 15 & 6
\end{pmatrix}$
satisfying the conditions of Lemma~\ref{orth_semi}.
\end{ex}

\begin{defn}\label{normalize} 
Let $H$ be an affine semigroup in $\Z^d$.
\begin{enumerate}
\item 
We put
\[
\overline H=\{ h\in G(H) \,|\, \exists n\in\N\setminus\{0\}, nh\in H\}
\]
and $\overline{H}$ is called the \wdef{normalization} of $H$.
Note that  
$\overline H$ is a semigroup containing $H$.
We say that $H$ is \wdef{normal}, if $H=\overline H$.

\item 
There exists $v_1, \dots, v_m\in\Z^d$
satisfying
\[
\overline H=\{ h\in G(H) \,|\, \langle h, v_i\rangle\geq0 \text{ for each }1 \le i \le m\},
\]
where $\langle\cdot,\cdot\rangle$ means the inner product
(cf. \cite[Proposition 6.1.2]{BH}).
We define the \wdef{relative interior} of $H$ as
\[
\relint H=
\{ h\in G(H) \,|\, \langle h, v_i\rangle>0 \text{ for each }1 \le i \le m\}
\cap H.
\]
\end{enumerate}
\end{defn}

\begin{prop}\label{normal}
Let $(H, E)$ be an orthogonal semigroup of order $m$.
Then the following are equivalent:
\begin{enumerate}
\item
$H$ is normal,
\item
$\Ap(H, E)=\{ h\in H \,|\, \sigma_i(h)< m\text{ for each }1 \le i \le d\}$,
\item
For each $h\in\Ap(H, E)$,
$h+ h^{\vee}=\sum_{i\in\supp h}m\e_i$ where 
$\supp h=\{ 1 \le i \le d \,|\,\sigma_i(h)\ne0\}$
and $h^{\vee}$ is defined in Definition~\ref{type}.
\end{enumerate}
\end{prop}

\begin{proof}
We prove $(1) \Rightarrow (2)$.
Assume that there exists $u\in\Ap(H, E)$ and $1 \le i\le d$
with $\sigma_i(u)\geq m$.
Then $u-m\e_i\notin H$ and $H$ is not normal.

Conversely, assume $(2)$.
Let $h\in \overline H=G(H)\cap \N^d$.
Then there is $v\in\Ap(H, E)$ and $b\in\Z E$
with $h=v+b$.
Then, for each $1 \le i \le d$, we have $\sigma_i(b)>-m$ 
because we assume the condition $(2)$
and $b\in\N E$, thus $h\in H$.
This implies that $H$ is normal,
which is the condition $(1)$.

Next, we prove the equivalence $(2) \Leftrightarrow (3)$.
Assume $(2)$ and let $h\in\Ap(H, E)$.
Since $h+h^{\vee}\in\N H$,
we have $h+h^{\vee}=\sum_{i\in\supp h}m\e_i$ by $(2)$.
Hence $(2)$ implies $(3)$.
Conversely, assume the condition $(3)$
and let $h\in\Ap(H, E)$.
By $(3)$, we have
$\sigma_i(h)\leq\sigma_i(h+h^{\vee})=m$ if $i\in\supp h$.
Note that $i\in\supp h$ if and only if $i\in\supp h^{\vee}$.
This implies $(2)$.
\end{proof}

\section{Cohen-Macaulay semigroup rings}

From now on, $\fd$ denotes always a field.
For a semigroup $H$,
we define the semigroup ring of $H$ over $\fd$ 
as
\[
\fd[H]=\fd[t^h]_{h\in H}=\bigoplus_{h\in H}\fd\cdot t^h.
\]
Here, for each $h \in H$, $t^h$ denotes merely a symbol, and $t^h\cdot t^{h'}=t^{h+h'}$ for $h, h'\in H$.
If $\rho : H\to H'$ is a homomorphism of semigroups,
i.e., it satisfies 
$\rho(h+h')=\rho(h)+\rho(h')$ for $h, h'\in H$,
it induces a ring homomorphism $\widetilde\rho : \fd[H]\to \fd[H']$.
Further, if $\rho$ is injective (resp. surjective or bijective),
so is $\widetilde\rho$.
If $H\subset\N^d$,
the inclusion map
induces an injection
\[
\fd[H] \hookrightarrow \fd[\N^d]\cong \fd[t_1, t_2,\dots, t_d],
\]
where $\fd[t_1,t_2,\ldots , t_d]$ denotes the polynomial ring over $\fd$ with $n$ indeterminates.
Hence, if $H$ is an affine semigroup,
then we may identify $\fd[H]$ as a graded $\fd$-subalgebra
of a polynomial ring over $\fd$ under any grading.
Then,
we may regard
$t^h$ as $t_1^{\sigma_1(h)}t_2^{\sigma_2(h)}\cdots t_d^{\sigma_d(h)}$,
using multi-index notation.
Recall that we denote the $i$-th entry of $h$ in $\Z^d$
by $\sigma_i(h)$.
We use this notation throughout this paper.

\begin{defn}\label{defn_deg} Let $H$ be an affine semigroup in $\Z^d$.
\begin{enumerate}
\item We define a degree of a monomial  $t^h$ in $\fd[H]\subset \fd[\N^d]$
by
\[\deg t^h=\rho(h)\; {\rm where}\; \rho :\Z^d=\bigoplus_{i=1}^d\Z\e_i\to\Z, 
\rho\left(\sum_{i=1}^da_i\e_i\right)=\sum_{i=1}^da_i.\]
\item For $h\in H$, we also denote $\rho(h)$ by $\deg h$.  
Namely, $\deg h=\deg t^h$.
\end{enumerate}
\end{defn}

\begin{rem}
By definition,
the degree depends on the embedding of $H$ in $\N^d$.
If $\iota : \N^d\to\N^d$ is an injection,
then we get another degree for $H$
by composing the original embedding $H\hookrightarrow \N^d$ with $\iota$.
\end{rem}

We recall definitions of the Hilbert series and the multiplicity. 

\begin{defn}\label{def_Handmult}
Let $R=\bigoplus_{n\geq0}R_n$ 
be a Noetherian positively graded ring with $R_0 = \fd$ a field.
Put $\mm=\bigoplus_{n>0}R_n$ be the graded maximal ideal of $R$. 
For a graded $R$-module $M$,
we denote the $n$-th component of $M$ by $M_n$.
A formal power series
\[
P_M(t)=\sum_{n\in\Z}(\dim_{\fd}M_n)t^n\in \Z[[t, t^{-1}]]
\]
is called the \wdef{Hilbert series} of $M$.
If $M$ is a finitely generated $R$-module,
there uniquely exist $Q_M(t)\in\Z[t, t^{-1}]$
and positive integers $a_1, \dots, a_{\dim_R M}$
satisfying
\[
P_M(t)=\frac{Q_M(t)}{\prod_{i=1}^{\dim_R M}(1-t^{a_i})}
\]
and $Q_M(1)\ne0$.
We denote the \wdef{multiplicity} of $M$ 
with respect to an ideal $\qq$ of definition of $M$ 
(i.e., $\dim_{\fd}M/\qq M<\infty$) by $\mult(\qq, M)$,
Namely,
\[
\mult(\qq, M)=\lim_{n\to\infty}\frac{s!}{n^s}\dim_{\fd}M/\qq^nM.
\]
where $s=\dim_R M$
(cf. \cite[Proposition 4.6.2]{BH}).
Then the following is well-known:
\begin{itemize}
\item 
If $\bb$ is a reduction of $\qq$ with respect to $M$,
namely there exists $n>0$ with $\qq^{n+1}M = \bb\qq^nM$,
then $\mult(\bb, M)=\mult(\qq, M)$.
\item 
Suppose that 
there exists a parameter ideal $\qq=(x_1, \dots,$ $x_{\dim_R M})$ 
for $M$ such that $\deg x_i=a_i$ for each $1 \le i\le \dim_R M$. Then 
\[
\mult(\qq, M)=Q_M(1)=\lim_{t\to1}
\left(\prod_{i=1}^{\dim_R M}(1-t^{a_i})\right)P_M(t).
\]
\end{itemize}
We denote $\mult(\mm, M)$ by $\mult(M)$, 
and refer to it as the \wdef{multiplicity} of $M$.
\end{defn}

\begin{ex}
Let $H=\langle n_1, \dots, n_e\rangle$ 
be a numerical semigroup
generated by positive integers $n_1<\cdots<n_e$,
i.e., $H=\sum_{i=1}^e\N n_i$ and $\gcd(n_1, \dots, n_e)=1$.
Then 
\[
P_{\fd[H]}(t)
=\left(\sum_{h\in\Ap(H, n_1)}t^h\right)\cdot(1+t^{n_1}+t^{2n_1}+\cdots)
=\frac{\sum_{h\in\Ap(H, n_1)}t^h}{1-t^{n_1}}.
\]
Then $Q_{\fd[H]}(t)=\sum_{h\in\Ap(H, n_1)}t^h$
and $\mult(\fd[H])=n_1$,
since $(t^{n_1})$ is a reduction of the graded maximal ideal of $\fd[H]$,
\end{ex}

We summarize known results and include proof in the present paper for the sake of completeness.

\begin{thm}\label{symp_CM}
Let $(H, E)$ be a simplicial semigroup 
and put $\aa=(t^b \mid b\in E)$, the ideal generated by $\{t^b \mid b \in E\}$ in $\fd[H]$.
Then the following are equivalent:
\begin{enumerate}
\item
$\fd[H]$ is Cohen-Macaulay,
\item
$\fd[H]$ is free over $\fd[\N E]$,
\item
$P_{\fd[H]}(t)=\frac{\sum_{h\in\Ap(H, E)}t^{\deg h}}
{\prod_{b\in E}(1-t^{\deg b})}$,
\item
$\mult(\aa, \fd[H])=|\Ap(H, E)|$.
\end{enumerate}
In this case,
we have $\fd[H]=\bigoplus_{h\in\Ap(H, E)}\fd[\N E]\cdot t^h$.
\end{thm}

\begin{proof}
Put $R=\fd[H]$ and $R'=\fd[\N E]$. If $R$ is Cohen-Macaulay, 
then it has depth $d=\dim R'$. 
By the Auslander-Buchsbaum formula, 
$R$ is projective over the regular ring $R'$, and hence free over $R'$. 
This establishes the implication $(1) \Rightarrow (2)$. 

If $R$ is free over $R'$, 
then there exists $B \subset H$ 
such that $\{t^h \mid h\in B\}$ forms an $R'$-basis of $R$. 
It clearly follows that $\Ap(H, E) \subseteq B$. 
If $B \setminus \Ap(H, E) \ne \emptyset$, there exists some $h \in B$, $h' \in \Ap(H, E)$, and $h'' \in \N E$ such that $h = h' + h''$. 
This situation would lead to a non-trivial relation in $R$, 
contradicting the assumption that $\{t^h \mid h\in B\}$ is a basis. 
Therefore, $B = \Ap(H, E)$, which shows that $(2)$ implies $(3)$.

Next, assuming condition $(3)$ holds, 
we have $Q_R(t) = \sum_{h \in \Ap(H, E)} t^{\deg h}$ 
and consequently, $\mult(\aa, R) = Q_R(1) = |\Ap(H, E)|$. 
This leads to condition $(4)$. 
Finally, given that 
$$\dim_{\fd} R/\aa = |\Ap(H, E)|,$$
we conclude that condition $(4)$ implies the Cohen-Macaulayness of $R$, namely, we deduce that $(4) \Rightarrow (1)$. 
This completes the proof of Theorem \ref{symp_CM}.
\end{proof}

\begin{defn}\label{canonical}
Let $(H, E)$ be a simplicial semigroup.
Assume that $R=\fd[H]$ is Cohen-Macaulay.
Let
\begin{align*}
\omega_H
&=\left\{
-w+\sum_{b\in E}b+h \,\middle|\, w\in\Soc(H, E), h\in H
\right\}\subset G(H)
\intertext{and put}
\K_R&=\fd[t^h]_{h\in\omega_H}\subset\fd[G(H)].
\intertext{The $R$-module $\K_R$ is called 
the \wdef{graded canonical module} of $R$.
The minimal number of homogeneous generators of $\K_R$,
called the type of $R$, denoted by $\type(R)$,
coincides
with the type of $H$ (see Definition~\ref{type}).
We define the $a$-invariant of $R$ (\cite[Definition 3.1.4]{GW}) by}
a(R)&=-\min\{ n \,|\, (\K_R)_n\ne0\}\\
&=\max\{\deg h \,|\, h\in\Soc(H, E)\}-\sum_{b\in E}\deg b.
\end{align*}
\end{defn}

Indeed, the following lemma shows that the module $\K_R$ defined above is consistent with the canonical module as generally defined
(compare with \cite[Corollary 4.4.6]{BH}).

\begin{lem} 
Let $(H, E)$ be a simplicial semigroup.
Assume that $R=\fd[H]$ is Cohen-Macaulay.
Then $P_{\K_R}(t) = (-1)^dP_{\fd[H]}(t^{-1})$
where $d=\rank H$.  
\end{lem}

\begin{proof}  
The assertion is readily derived 
from the definition of $\K_R$ 
and the implication $(3) \Rightarrow (1)$ in Theorem \ref{symp_CM}.
\end{proof}

Note that 
the concept of the graded canonical module 
is defined in a more general setting using the homological method. 
For detailed information, please refer to \cite[Definition 3.6.8]{BH}.

For $w\in \Soc(H, E)$, there exists $w^{\vee}\in\Ap(H, E)$
with $w^{\vee}\equiv -w \mod \Z E$
(see Definition~\ref{type}).
Hence, there exists $v\in H$
with $-w+v\in H$ for any $w\in \Soc(H, E)$.
This implies that
we may regard the graded canonical module
as an ideal of $R$
by suitable shifting
(cf. \cite{EM2023}).
Further if $H$ is normal,
$\omega_H$
is the relative interior of $H$
defined in Definition~\ref{relint}
(see \cite[Theorem 6.3.5]{BH}).

Now, we aim to extend \cite[Theorem 3.2]{HJS}
to the case when $(H, E)$ is not necessarily normal.
We begin with the following computation on the multiplicity of $\K_R/t^vR$
with respect to the ideal $\bb = (t^b-t^{b'} \mid b, b' \in E)$
where $R=\fd[H]$
and $v\in\omega_H$.
Notice that, 
for fixed $b \in E$, 
$\bb = (t^b -t^{b'} \mid b \ne b' \in E)$. 
Hence, $\bb$ is minimally generated by $d-1$ elements, 
where $d$ is the rank of $H$. 
Hence $\bb$ is a parameter ideal for $\K_R / t^v \K_R$ for $v \in \omega_H$.

\begin{prop}\label{Hilbert}
Let $(H, E)$ be an orthogonal semigroup of rank $d$ and order $m$.
Assume that $R=\fd[H]$ is Cohen-Macaulay
and put $\bb=(t^b-t^{b'} \mid b, b'\in E)$.
Then, we have the following.
\begin{enumerate}
\item
Let $w\in\Soc(H, E)$ and $v=-w+\sum_{b\in E}b\in\omega_H$. 
\begin{enumerate}
\item
For each $h\in\Ap(H, E)$,
there uniquely exist $h'\in\Ap(H, E)$, $u\in\N E$, and $c_{w, h}\geq 0$
such that $h+h'=w+u$
and $m{\cdot}c_{w, h}=\deg u=\deg h+\deg h'-\deg w$.
\item
For all $h \in \Ap(H,E)$, 
the equality $c_{w, h}=0$ holds if and only if $h\leq_H w$.
\item 
The following equality holds:
\[
\mult(\bb, \K_R/t^vR)
=\sum_{\begin{subarray}{c}h\in\Ap(H, E)\\ h\not\leq_Hw\end{subarray}}c_{w, h}.
\]
\item
We always have $\mult(\bb, \K_R/t^vR)
\geq\mult(\mm, \K_R/t^vR)\geq\type H-1$.
Furthermore, the equality $\mult(\bb, \K_R / t^vR) = \type H-1$ holds 
if and only if 
\[
\Ap(H, E)\setminus\{ h\in\Ap(H, E) \mid h\leq_Hw\}
=\Soc(H, E)\setminus\{w\}
\]
and $c_{w, h}=1$ for all $h\in\Soc(H, E)\setminus\{w\}$.
\end{enumerate}
\item
Let $v, v'\in\omega_H$. If $\deg v\geq\deg v'$,
then
\[
\mult(\bb, \K_R/t^vR)
\geq
\mult(\bb, \K_R/t^{v'}R)
\]
and the equality holds if and only if $\deg v=\deg v'$.
\end{enumerate}
\end{prop}

\begin{proof}
We prove $(1)$.
Since $w-h\in G(H)$,
and since $R$ is Cohen-Macaulay,
there uniquely exists $u\in\N E$ with $w-h+u\in\Ap(H, E)$
by the definition of Ap\'ery sets
(see Definition~\ref{type}).
Then, by putting $h'=w-h+u$ and $c_{w, h}=\frac{\deg u}m$,
we get the assertion (a).
Further,
$c_{w, h}=0$ if and only if $w-h=h'\in\Ap(H, E)\subset H$,
namely $h\leq_Hw$.
This proves (b).

(c)\ 
Put $\deg v=s$.
Note $\dim_R \K_R/t^vR=d-1$
and that $\bb$ is generated by $d-1$ elements.
Since
\[\begin{CD}
0 @>>> R @>t^v>> \K_R @>>> \K_R/t^vR @>>> 0
\end{CD}\]
is exact,
we have
\begin{align*}
P_{\K_R/t^vR}(t)
&=P_{\K_R}(t)-t^sP_R(t)
=(-1)^dP_R(t^{-1})-t^sP_R(t)\\
&=\frac{t^s(t^{\deg w}\sum_{h'\in\Ap(H, E)}t^{-\deg h'}
-\sum_{h\in\Ap(H, E)}t^{\deg h})}
{(1-t^m)^d}\\
&=\frac{t^s\sum_{h\in\Ap(H, E)}t^{\deg w-\deg h'}(1-t^{\deg h+\deg h'-\deg w})}
{(1-t^m)^d}\\
&=\frac{\sum_{h\in\Ap(H, E)}t^{s+\deg w-\deg h'}(1-t^{mc_{w, h}})}
{(1-t^m)^d},
\end{align*}
by Theorem~\ref{symp_CM}.
Note that in the second equation from the bottom, 
for each $h\in\Ap(H, E)$,
we choose $h'$ so that
it satisfies the condition in (a).
Hence,
by Definition~\ref{def_Handmult},
\begin{align*}
\mult(\bb, \K_R/t^vR)
&=\lim_{t\to1}(1-t^m)^{d-1}P_{\K_R/t^vR}(t)\\
&=\lim_{t\to1}\frac{\sum_{h\in\Ap(H, E)}t^{s+\deg w-\deg h'}(1-t^{mc_{w, h}})}
{1-t^m}\\
&=\sum_{h\in\Ap(H, E)}c_{w, h}.
\end{align*}
Therefore (c) is proved, and (d) follows from (c).

(2)\ 
Put $\deg v=s$ and $\deg v'=s'$. Then $s \ge s'$ and, by Definition~\ref{def_Handmult},
we have
\begin{align*}
\lefteqn{\mult(\bb, \K_R/t^vR)-\mult(\bb, \K_R/t^{v'}R)}\\
&=\lim_{t\to1}(1-t^m)^{d-1}(P_{\K_R/t^vR}(t)-P_{\K_R/t^{v'}R}(t))\\
&=\lim_{t\to1}(1-t^m)^{d-1}((P_{\K_R}(t)-t^sP_R(t))
-(P_{\K_R}(t)-t^{s'}P_R(t)))\\
&=\lim_{t\to1}t^{s'}(1-t^{s-s'})(1-t^m)^{d-1}P_R(t)\\
&=\lim_{t\to 1}t^{s'}\frac{1-t^{s-s'}}{1-t^m} (1-t^m)^d P_R(t)\\
&=\frac{s-s'}m\mult(\bb+t^bR, R)\geq0,
\end{align*}
where $b\in E$.
\end{proof}

\begin{rem}
The equality $c_{w, h} = 1$ does not always hold, even if $h\not\leq_Hw$.
Let $H=\langle6, 7, 16, 17\rangle$ be a numerical semigroup,
and $w=21, h=17\in\Soc(H, E)$,
where $E=\{6\}$.
Then $17+16-21=12=6\cdot2$ and $c_{6, 17}=2$.
\end{rem}

As previously mentioned, our focus is on the module $\K_R/t^vR$, where $v \in \omega_H$. 
In order to explore the structure of $\K_R / t^vR$, particularly to determine the conditions under which $\bb = (t^b - t^{b'} \mid b, b' \in E)$ is a reduction of $\mm$ with respect to $\K_R/ t^vR$, we first examine an analogous problem of when $\bb$ is a reduction of $\mm$ with respect to $R/t^wR$ for each $w \in H$. 
While \cite[Theorem 2.4]{HJS} has essentially proven the following proposition under the assumption that $H$ is normal, we will also provide proof because this paper aims to further the discussion without making this assumption.

\begin{prop}[{cf. \cite[Theorem 2.4]{HJS}}]\label{almoststandard}
Let $(H, E)$ be an orthogonal semigroup of rank $d$ and order $m$.
Put $R=\fd[H]$,
$\mm$ the graded maximal ideal of $R$
and $\bb = (t^b-t^{b'} \mid b,b' \in E)$.
Then, for $w\in H\setminus\{0\}$,
the following are equivalent:
\begin{enumerate}
\item
$\min\{\deg h \,|\, h\in H\setminus\{0\}, \supp w\not\subset\supp h\}
\geq m$ where $\supp w=\{ 1 \le i\le d \,|\, \sigma_i(w)\ne0\}$.
\item
$\bb$ is a reduction of $\mm$ with respect to $R/t^wR$.
\end{enumerate}
\end{prop}

\begin{proof}
Assume condition (2). Then we can find $n > 0$ such that $\mm^{n+1} + t^wR = \bb\mm^n + t^wR$.
Suppose the existence of an element $h \in H \setminus \{0\}$ for which $\supp w \not\subset \supp h$ and $\deg h < m$. 
We then select such an element $h$, ensuring that it has the smallest possible degree, denoted as $\deg h$. 
Consequently, $t^{(n+1)h} \in \mm^{n+1} \subset \bb\mm^n + t^wR$. 
We write $t^{(n+1)h} = x + t^wy$ for some $x \in \bb\mm^n$ and $y \in R$. Then,  thanks to $\supp w \not\subset \supp h$, we have $t^{(n+1)h} \notin t^wR$, which implies $x \ne 0$.
Furthermore, it is possible to find $h_1, \dots, h_n \in H \setminus \{0\}$ and $b \in E$ such that
$$
(n+1)h = h_1 + \cdots + h_n + b.
$$
For any $1 \le j \le n$, we deduce that $\supp w \not\subset \supp h_j$; otherwise, it would mean $\supp w \subset \supp h_j \subset \supp h$, contradicting our choice of $h$ due to the minimality of $\deg h$. 
Therefore, $\deg h_j \geq \deg h$. This leads to 
$$
n\deg h + m \leq \sum_{j=1}^n \deg h_j + m = \deg (n+1)h = (n+1)\deg h,
$$
which contradicts the selection of $h$. Consequently, condition (1) is established.

Conversely, assume condition (1). We claim the following.

\begin{claim}
For every $h \in H \setminus \{0\}$, 
there exists $n > 0$ such that $t^{(n+1)h} \in \bb \mm^n + t^wR$.
\end{claim}

The above claim implies condition (2), 
namely, $\bb$ is a reduction of $\mm$ with respect to $R/t^wR$.

\begin{proof}[Proof of Claim]
Let $h \in H \setminus \{0\}$. 
Initially, 
let's examine the case where $\supp w \subset \supp h$. 
In this case, 
we can identify $n > 0$ 
such that $\sigma_i(nh - w) \geq 0$ for each $1 \le i \le d$. 
Consequently, 
this allows for the existence of $n' > 0$ for which $n'(nh - w) \in H$, 
leading to $n'nh \in w + H$, and thus $t^{n'nh} \in t^wR$.

Next, 
we consider the case where $\supp w \not\subset \supp h$. 
By condition (1), we know that $\deg h \ge m$.
Since $H$ is simplicial, 
we have $nh \in \N E$ for sufficiently large $n$.
Recalling that $w + w^{\vee} \in \N E$ (see Definition \ref{type}), 
if $\deg nh>\deg(w + w^{\vee})$, then
$t^{nh} \in \bb + t^wR$. 
Subsequently, 
we may select an integer $n > 0$ ensuring both 
$t^{nh} \in \bb + t^wR$ and $nh \in \N E$. 
Since $\deg h \ge m$, 
we have $\deg nh = cm$ for some $c \ge n$. 
Therefore, we deduce that $t^{nh} \in \bb\mm^{n-1} + t^wR$.
\end{proof}
This completes the proof of Proposition \ref{almoststandard}.
\end{proof}

\begin{defn}\label{def_slim}
Let $(H,E)$ be an orthogonal semigroup.
If the conditions of Proposition~\ref{almoststandard}
are satisfied for any $w\in H\setminus\{0\}$, 
we say that $H$ is a \wdef{slim} semigroup,
after the definition in \cite{HJS}.
\end{defn}

\begin{rem}\label{rem_slim}
By this definition,
$H$ is slim if and only if
$\bb = (t^b-t^{b'} \mid b,b' \in E)$ 
is a reduction with respect to $R/t^wR$
for $w\in\relint H$.
This is also equivalent to the condition
$\min\{\deg h \,|\, h\ne0, h\in H\setminus\relint H\}\geq m$
holds.
Hence any orthogonal semigroup of rank $2$ is slim.
It should be noted that the assumption `orthogonal' is not included in the original definition by \cite{HJS}. However, we will later demonstrate that this difference is not significant in Proposition \ref{orth_semi2}.
\end{rem}

\begin{ex}
Let $H$ be a simplicail semigroup defined in \cite[Example 2.8]{HJS}.
Applying Lemma~\ref{orth_semi} to $H$,
we have an orthogonal semigroup
\[
H'=\begin{pmatrix}
5 & 0 & 0 & 1 & 2 \\
0 & 5 & 0 & 3 & 1 \\
0 & 0 & 5 & 0 & 0
\end{pmatrix},
\] isomorphic to $H$.
Then
\[
\min\{\deg h \,|\, h\ne0, h\in H'\setminus\relint H'\}=3<5
\]
implies that $H'$ is not slim (see Remark~\ref{rem_slim}).
\end{ex}

\begin{prop}[{cf. \cite[Theorem 3.2]{HJS}}]\label{slim}
Let $(H, E)$ be a slim semigroup 
such that $R=\fd[H]$ is Cohen-Macaulay with the graded maximal ideal $\mm$.
We put $\bb=(t^b-t^{b'} \mid b,b' \in E)$ and take $v\in\omega_H$.
Then $\bb$ is a reduction of $\mm$ with respect to $\K_R/t^vR$,
thus $\mult(\K_R/t^vR)=\mult(\mm, \K_R/t^vR) =\mult(\bb, \K_R/t^vR)$
(note that the multiplicity is defined in  Definition~\ref{def_Handmult}).
\end{prop}

\begin{proof}
We first note that $\bb$ is a reduction of $\mm$ 
with respect to $I / (t^wR \cap I)$ 
for every ideal $I$ of $R$ and every $w \in \relint H$. 
It follows from \cite[Exercise 4.6.11]{BH},
since $I / (t^wR \cap I)\subset R/t^wR$
and since $\bb$ is a reduction of $\mm$ with respect to $R/t^wR$
by Proposition~\ref{almoststandard}.
To prove the proposition, 
consider the monomial ideal $I = (t^h)_{h \in u + \omega_H}$ for some $u \in H$ such that 
\[
u + \omega_H = u + \bigcup_{w \in \Soc H} \left(\left(-w + \sum_{b \in E}b\right) + H\right) \subseteq \relint H.
\]
Then $\bb$ is a reduction of $\mm$ 
with respect to $I / t^{u+v}R \cong t^u \K_R / t^{u+v}R \cong \K_R/t^vR$ 
as desired.
\end{proof}

\section{AG semigroups}

\subsection{A short survey on almost symmetric numerical semigroups and Nari's characterization}

In this subsection, we recall the almost symmetry in numerical semigroups introduced by V. Barucci-R. Fr\"{o}berg \cite{BF}.

Let $H$ be a numerical semigroup, which is a simplicial semigroup of rank $1$.
For $0 < h \in H$, the Ap\'ery set $\Ap(H,h)$ of $H$ with respect to $h$ is defined as
\[
\Ap(H,h) = \{h' \in H \,|\, h' -h \notin H\}.
\]
This is exactly the same as $\Ap(H, \{ h \})$ in Definition \ref{type}.

An integer $f$ is called a \wdef{pseudo-Frobenius number} of $H$, 
if $f \notin H$ but $f + h \in H$ for all $0 < h \in H$. 
We put 
\[\PF(H) = \{f \in \mathbb{Z} \,|\, 
\text{$f$ is a pseudo-Frobenius number of $H$}\}.
\]
It is well-known that, for $0 < h \in H$, 
\[
\PF(H) = \{h' - h \,|\, 
\text{$h'$ is maximal in $\Ap(H,h)$ with respect to the order $\leq_H$}
\}.
\]
The maximum number of $\PF(H)$ is called 
the \wdef{Frobenius number} of $H$, denoted by $\Fr(H)$.
Notice that, in this case, 
\[
\Soc(H,\{h\}) = \{f + h \,|\, f \in \PF(H)\}
\] 
for each $h \in H$. 
Therefore we get $\omega_H = \{-f + h\,|\, f \in \PF(H), h \in H\}$.

\begin{defn}[\cite{BF}]
We say that $H$ is \wdef{almost symmetric}, 
if $h + x + \Fr(H) \in H$ for all $h \in H$ and $x \in \omega_H$.
\end{defn}

\begin{rem}
Suppose that $H$ is symmetric, 
i.e., for all $z \in \Z$, $z \in H$ 
if and only if $\Fr(H) - z \notin H$.
Then, we get clearly that 
$\omega_H= \{-\Fr(H) + h \,|\, h \in H\}$. 
Hence $H$ is almost symmetric. 
\end{rem}

The following is a characterization of 
the almost symmetry of numerical semigroup rings 
in terms of $\PF(H)$ given by H. Nari.

\begin{thm}[{\cite[Theorem 2.4]{Nari13}}]\label{Nari}
Let $\PF(H) = \{f_1 < f_2 < \cdots < f_r = \Fr(H)\}$. 
The following are equivalent:
\begin{enumerate}
\item 
$H$ is almost symmetric,
\item 
$f_i + f_{r-i} = f_r$ for all $1 \le i \le r-1$.
\end{enumerate}
\end{thm}

Later we will give a generalization of Theorem \ref{Nari}.

\subsection{Almost Gorenstein property for graded rings}

We recall the definition of 
\AGG rings
given by \cite{GTT}, under our general settings.

\begin{defn}[{cf. \cite[Definition 2.1]{GTT}}]\label{Ulrich module}
Let $R = \bigoplus_{n \ge 0} R_n$ be 
a positively graded Cohen-Macaulay ring with $R_0=k$ a field. 
Let $\mm$ be the graded maximal ideal of $R$ 
and $M$ be a finitely generated graded $R$-module.
We say that $M$ is an \wdef{Ulrich $R$-module}, 
if $M$ is a Cohen-Macaulay graded module and 
generated by $\mult(\mm, M)$ homogeneous elements.
Note that we do not assume $\dim_R M=\dim R$ and that
the concept of multiplicity $\mult(M)=\mult(\mm, M)$ is defined in Definition~\ref{def_Handmult}.
\end{defn}

\begin{rem}
Notice that the notion of Ulrich modules in \cite[Definition 2.1]{GTT} 
was defined for finitely generated modules over Noetherian local rings. 
Definition \ref{Ulrich module} is a rephrasing of 
\cite[Definition 2.1]{GTT} for the graded version. See \cite{HJS} also.
\end{rem}

\begin{defn}[{\cite[Definition 8.1]{GTT}}]\label{AGG}
Let $R = \bigoplus_{n \ge 0}R_n$ be a Cohen-Macaulay graded ring 
with $k=R_0$ a field. 
Assume that $R$ possesses the graded canonical module $\K_R$. 
Let $\mm$ be the graded maximal ideal of $R$.
We say that $R$ is an 
\wdef{\AGG ring},
if there exists an exact sequence 
\[
0 \to R \to \K_R(-a(R)) \to C \to 0
\]
of graded $R$-modules such that $C$ is an Ulrich $R$-module.
This is equivalent to the existence of a homogeneous element $x\in\K_R$ 
whose degree is $-a(R)$ and $\mult(\mm, \K_R/xR) = \type(R)-1$. 
\end{defn}

Let $R$ be an \AGG ring 
and $\mm$ the graded maximal ideal of $R$. 
Then $R_{\mm}$ is an almost Gorenstein local ring.
The definition of almost Gorenstein local rings 
is also given in \cite{GTT} for Cohen-Macaulay local ring $A$, 
which, in the context of this paper, 
can be characterized by the existence of a specific element $x \in \K_A$, 
where $\K_A$ denotes the canonical module of $A$.
More explicitly, 
$\K_A/xA$ is an Ulrich module in the sense of \cite[Definition 2.1]{GTT}. 
For further detail, we refer to \cite[Section 3]{GTT}).

The concept of an Ulrich element $v \in \omega_H$ is defined in \cite{HJS} if $H$ is normal. 
We extend this definition to the case $H$ is not necessarily normal.

\begin{defn}\label{def_Uelement}
Let $(H, E)$ be a simplicial semigroup
and $\bb=(t^b-t^{b'} \mid b, b'\in E) \subset R=\fd[H]$.
Assume that $R$ is Cohen-Macaulay.
An element $v\in\omega_H$ is called an \wdef{Ulrich element},
if the equality 
\[
\mult(\bb, \K_R/t^vR)=\type H-1
\]
holds, where  $\type H$ is the type of $H$  (see Definition~\ref{type}) and
recall that $\type H=\type (R)$ holds
(Definition~\ref{canonical}).
We say that {\it $\omega_H$ has an Ulrich element} 
if an element   $v \in \omega_H$ is an Ulrich element.\par
If $H$ is normal, then this definition agrees with that of 
\cite[Definition 3.1]{HJS} 
since $\omega_H=\relint H\subset H$ if $H$ is normal.
\end{defn}

\begin{rem}
In the case $H$ is normal, 
according to \cite[Definition 3.1]{HJS},  
an element $v \in \omega_H$ is an Ulrich element if the equation
 $$\mult(\mm, \K_R/t^vR) = \type H - 1$$ is satisfied. 
 In the following theorem, we establish that an element $v \in \omega_H$
  is an Ulrich element in the sense of \cite{HJS}, provided $v$ is an 
  Ulrich element in our sense. 
The inverse implication holds when $H$ is a slim semigroup.
\par
We have no example of $v \in \omega_H$ such that $\K_R/t^vR$ is an Ulrich module
 and $\mult(\bb, \K_R/t^vR) > \mult(\mm, \K_R/t^vR)$.
If $H$ is such an example,
then the rank of $H$ must be greater than two,
since any affine semigroup of rank two is slim
(see Definition~\ref{def_slim}).
\end{rem}

We are now ready to prove the main theorem, 
which is a generalization of \cite[Theorem 3.2]{HJS} 
for the case where $H$ is not necessarily normal.

\begin{thm}\label{slim2}
Let $H$ be a simplicial semigroup such that $R=\fd[H]$ is Cohen-Macaulay.
If $\omega_H$ has an Ulrich element $v$, then $\K_R/t^vR$ is an Ulrich module.
The converse is true, if $H$ is slim.
\end{thm}

\begin{proof}
If $\omega_H$ has an Ulrich element $v$, thanks to Proposition~\ref{Hilbert} $(1)$ (d), we have
\[
\type H-1\leq\mult(\mm, \K_R/t^vR)\leq\mult(\bb, \K_R/t^vR)=\type H-1
\]
and hence $\K_R/t^vR$ is an Ulrich module
(note that the multiplicity is defined in  Definition~\ref{def_Handmult}).
Conversely, suppose that $H$ is slim.
Let $v \in \K_R$ such that $K/t^vR$ is an Ulrich $R$-module. 
Then, since $\mult(\mm, \K_R/t^vR)=\mult(\bb, \K_R/t^vR)$ 
by Proposition \ref{slim}, $v$ is an Ulrich element.
\end{proof}

\begin{prop}\label{orth_semi2}
Let $H$ be a simplicial semigroup
such that $\fd[H]$ is Cohen-Macaulay.
Then the following are equivalent:
\begin{enumerate}
\item
$\omega_H$ has an Ulrich element,
\item
there exists an orthogonal semigroup $H'$ isomorphic to $H$,
such that $\omega_{H'}$ has an Ulrich element
whose degree is $-a(\fd[H'])$.
\end{enumerate}
\end{prop}

\begin{proof}
To prove $(1)$ implies $(2)$, assume that $\omega_H$ has an Ulrich element $v$.
Let us consider the map $\iota : \Z^d\to\Z^d$ as defined in the proof of Lemma~\ref{orth_semi}, and set $H' = \iota(H)$.
Then, since $\iota$ induces a $\fd$-algebra isomorphism from $\fd[H]$ to $\fd[H']$,
we have $\K_{\fd[H']}/t^{\iota(v)}\K_{\fd[H']}
\cong\K_{\fd[H]}/t^{v}\K_{\fd[H]}$, 
and hence $\iota(v)$ is an Ulrich element.
We now prove $\deg\iota(v)=-a(\fd[H'])$.
Put $R=\fd[H']$ and let $\bb$ be the ideal defined in Proposition~\ref{Hilbert}.
Taking $w\in\omega_{H'}$ with $\deg w=-a(R)$ and applying Proposition~\ref{Hilbert} $(2)$, since $-a(R) = \deg w \ge \deg \iota(v)$
we have
\[
\type H'-1
\leq\mult(\bb, \K_R/t^wR)
\leq\mult(\bb, \K_R/t^{\iota(v)}R)
=\type H'-1,
\]
which clearly implies $\mult(\bb, \K_R/t^wR) = \mult(\bb, \K_R/t^{\iota(v)}R)$. 
Consequently, again by Proposition~\ref{Hilbert} $(2)$, we have $\deg\iota(v)=\deg w=-a(R)$ as wanted.
The converse implication can be proven by a similar reasoning.
\end{proof}

\begin{defn}
Let $H$ be a simplicial semigroup
such that $\fd[H]$ is Cohen-Macaulay.
We say that $H$ is an \wdef{AG semigroup},
if the equivalent conditions in Proposition~\ref{orth_semi2} are satisfied.
\end{defn}

\begin{rem}\label{orth_semi3}
By Proposition~\ref{orth_semi2},
if a simplicial semigroup $H$ is an AG semigroup,
then there exists a semigroup $H'$ isomorphic to $H$
such that $\fd[H']$ is an \AGG ring.
Further if $H$ is orthogonal,
then $\fd[H]$ is an \AGG ring,
by the proof of Proposition~\ref{orth_semi2}.
Further assume that $\fd[H]$ has a standard grading,
i.e. generated by the elements of the same degree.
By \cite[Corollary 2.7]{Hi},
$\fd[H]$ has an Ulrich element
in the sense of \cite{HJS},
if it is an \AGG ring.
Therefore,
by Theorem~\ref{slim2},
$H$ is an AG semigroup if and only if
$\fd[H]$ is an \AGG ring,
in this case.
\end{rem}
The following example shows that
the degree of Ulrich elements 
is not necessarily $-a(\fd[H])$
if $H$ is not orthogonal.

\begin{ex}[{\cite[Example 8.8]{GTT}}]\label{GTT_ex}
Let
\[
H=
\begin{pmatrix}
1 & 3 & 3 & 3 \\
0 & 3 & 1 & 2
\end{pmatrix}
\quad
\text{and}
\quad
E=
\begin{pmatrix}
1 & 3  \\
0 & 3 
\end{pmatrix}.
\]
Then we know that $R=\fd[H]$ is not an \AGG  ring
by \cite[Example 8.8]{GTT}.
Note that $\binom12=-\binom31+\binom10+\binom33\in\omega_H$ 
is an Ulrich element,
whose degree is $3>-a(R)=2$.
Moreover, we can check that $\Soc(H, E)=\left\{\binom31, \binom32\right\}$
and $2\binom32=\binom31+\binom33$. Hence $c_{\binom31, \binom32}=1$.
Consequently, we get the equality
\[
\mult(\bb, \K_R/t^{\binom12}R)
=c_{\binom31, \binom32}=1=\type H-1
\]
(see Proposition~\ref{Hilbert} (1) (c)), 
which implies that $H$ is an AG semigroup.

Now, to make the homogeneous form of $H$, we consider the injection
$\iota : \Z^2\to\Z^2$ defined by the matrix
$\begin{pmatrix}3 & -3 \\ 0 & 1\end{pmatrix}$.
Then we obtain 
\[
H'=
\begin{pmatrix}
3 & 0 & 6 & 3 \\
0 & 3 & 1 & 2
\end{pmatrix}
\quad
\text{and}
\quad
E=
\begin{pmatrix}
3 & 0  \\
0 & 3 
\end{pmatrix}.
\]

$\iota\binom12=\binom32$ is an Ulrich element in $\iota(H)$
and $\fd[H']$ is an \AGG ring by Proposition~\ref{orth_semi2}.
We note that
the property
that the semigroup ring is
an \AGG ring
depends on the grading,
which is also pointed out in \cite[Remark 3.2]{EM2023}.
\end{ex}

\begin{rem}
Consider the following assertions.
\begin{enumerate}
\item
$R_{\mm}$ is an almost Gorenstein local ring,
\item 
$R$ is an \AGG ring 
(Definition~\ref{AGG}),
\item
$H$ is an AG semigroup,
\end{enumerate}
Then, in general, $(2) \Rightarrow (1)$ and $(3) \Rightarrow (1)$ hold
(see Proposition~\ref{orth_semi2}). 
However,
Example~\ref{GTT_ex} shows that
the condition $(1)$ does not implies $(2)$.
Later,
we see that $(2)$ does not always imply $(3)$ in Example~\ref{notmonomial}.
Hence $(1)$ does not implies $(3)$, too.

If $H$ is a numerical semigroup,
then the above three conditions are equivalent, 
and furthermore, 
the condition that $H$ is almost symmetric 
can also be added as an equivalent condition. 
Notice that any numerical semigroup $H$ is 
slim by Proposition~\ref{almoststandard}.
\end{rem}

The following theorem is a generalization of Theorem \ref{Nari}.

\begin{thm}\label{simplicialAG}
Let $(H, E)$ be a simplicial semigroup
such that $R=\fd[H]$ is Cohen-Macaulay.
Then $H$ is an AG semigroup
if and only if
there exists $w\in\Soc(H, E)$
such that
the following two conditions are satisfied:
\begin{enumerate}
\item For every $h \in \Ap(H,E)$, 
$h\in\Soc(H, E)\setminus\{w\}$
if and only if $h\not\leq_Hw$ and,
\item
for each $h\in\Soc(H, E)\setminus\{w\}$,
there exist $b\in E$ and $h'\in\Soc(H, E)$
such that $h+h'=w+b$.
\end{enumerate}
When the above conditions are satisfied,
$v=-w+\sum_{b\in E}b\in\omega_H$ is an Ulrich element of $H$.
\end{thm}

\begin{proof}
By Proposition~\ref{orth_semi2}, we may assume that $H$ is an orthogonal semigroup of order $m$.
Throughout this proof, we maintain the notation $c_{w,h}$ as in Proposition~\ref{Hilbert} $(1)$ (a) for each $w \in \Soc(H,E)$ and $h \in \Ap(H,E)$.
Take $w\in \Soc(H, E)$ such that $\deg w\geq\deg h$ for all $h\in\Soc(H, E)$, i.e., $w$ has the maximum degree in $\Soc(H,E)$.
Proposition~\ref{Hilbert} $(1)$ (d) establishes that $H$ is an AG semigroup if and only if
\[
\Ap(H, E)\setminus\{ h\in\Ap(H, E) \,|\, h\leq_Hw\}
=\Soc(H, E)\setminus\{w\}
\]
and $c_{w, h}=1$ for $h\in\Soc(H, E)\setminus\{w\}$.
This condition is equivalent to that, for each $h\in\Soc(H, E)\setminus\{w\}$,
there exist $h'\in\Soc(H, E)$ and $b\in E$ such that $h+h'=w+b$. 
Thus, the proof is complete.
\end{proof}

\begin{cor}
Let $(H_1, E_1)$ and $(H_2, E_2)$ be simplicial semigroups
of rank $r_1$ and $r_2$, respectively.
Assume that both $\fd[H_1]$ and $\fd[H_2]$ are Cohen-Macaulay.
The product $H=H_1\times H_2\subset\Z^{r_1+r_2}$ is an AG semigroup,
if and only if either
\begin{enumerate}
\item
$\type H_1= \type H_2=1$, or
\item
$H_1$ (resp. $H_2$) is an AG semigroup and $H_2=\N^{r_2}$
(resp. $H_1=\N^{r_1}$).
\end{enumerate}
\end{cor}

\begin{proof}
Put $E=E_1\cup E_2\subset\Z^{r_1+r_2}$.
Note that $\fd[H]$ is Cohen-Macaulay,
\begin{align*}
\Ap(H, E)&=\Ap(H_1, E_1)\times\Ap(H_2, E_2)
\intertext{and}
\Soc(H, E)&=\Soc(H_1, E_1)\times\Soc(H_2, E_2)
\end{align*}
by definition.
Hence, by Theorem~\ref{simplicialAG},
if the type of $H_1$ and that of $H_2$ are one,
or if
$H_1$ is an AG semigroup and $H_2=\N^{r_2}$,
then $H$ is an AG semigroup
(note $\Ap(\N^{r_2}, E_2)=\Soc(\N^{r_2}, E_2)=\{0\}$),
indeed.
Assume that $H$ is an AG semigroup
and that the condition $(1)$ is not satisfied.
Then we may assume $|\Soc(H_1, E_1)|>1$.
By Theorem~\ref{simplicialAG},
there exist $w_1\in\Soc(H_1, E_1)$
and $w_2\in\Soc(H_2, E_2)$
such that $w=(w_1, w_2)$ satisfies 
the conditions of Theorem~\ref{simplicialAG}.
Since $|\Soc(H_1, E_1)|>1$,
there is $h_1\in\Soc(H_1, E_1)$ with $h_1\ne w_1$.
Applying the condition $(2)$ in Theorem~\ref{simplicialAG}
to $(h_1, w_2)\in\Soc(H, E)$,
there exist $h_1'\in\Soc(H_1, E_1)$,
$h_2'\in\Soc(H_2, E_2)$
and $b\in E$
satisfying
\[
(h_1, w_2)+(h_1', h_2')=(w_1, w_2)+b.
\]
Then $b\in E_1$ and $0=h_2'\in\Soc(H_2, E_2)$,
which implies $H_2=\N^{r_2}$ and
$\Soc(H_2, E_2)=\{0\}$.
And $H_1$ is an AG semigroup by Theorem~\ref{simplicialAG}.
\end{proof}

Further assume that $H$ is slim,
for example the case $\rank H_1=\rank H_2=1$.
Then the above corollary says that,
if $\fd[H]$ is almost Gorenstein ring,
then both $\fd[H_1]$ and $\fd[H_2]$ are Gorenstein,
or $\fd[H]$ is the polynomial extension of 
an almost Gorenstein ring
$\fd[H_1]$ or $\fd[H_2]$
(cf. \cite[Theorem 8.5]{GTT}).
This question is also discussed in  \cite{MM}.
Note that $H$ is not necessarily slim,
even when both $H_1$ and $H_2$ are slim.

\begin{ex}
Assume that $H$ is a numerical semigroup.
Then $E=\{h\}$ where $h$ is the minimal positive number in $H$, 
called the multiplicity of $H$.
If $w\in\Soc(H)$ satisfies the conditions (1) and (2) 
in Theorem~\ref{simplicialAG},
then $w=\Fr(H)+h$ where $\Fr(H)$ is the Frobenius number of $H$.
If $f\in\PF(H)\setminus\{\Fr(H)\}$, 
then $f+h\in\Soc(H,E)$ and, by Theorem~\ref{simplicialAG} $(2)$
there exists $f'\in\PF(H)$
\[
(f+h)+(f'+h)=(\Fr(H)+h)+h,
\]
thus $f+f'=\Fr(H)$.
This shows that Theorem~\ref{simplicialAG} 
is an extension of Nari's Theorem (Theorem~\ref{Nari}).
\end{ex}

\begin{ex}[{\cite[Example 4.6]{HJS}}]
Let
\[
H=\begin{pmatrix} 11 & 5 & 4 & 3 \\ 13 & 6 & 5 & 4 \end{pmatrix}
\quad\text{and}\quad
E=\begin{pmatrix} 11 & 3 \\ 13 & 4 \end{pmatrix}.
\]
Then $(H, E)$ is a normal simplicial semigroup
and
\[
2\binom45=\binom56+\binom34
\]
implies that $H$ is an AG semigroup.
Further, $\fd[H]$ is an \AGG ring,
since $\rank H=2$.
\end{ex}

\begin{ex}[{\cite[Example 10.8]{GTT}}]
We consider the Veronese subring of a polynomial ring $\fd[X_1, \dots, X_d]$.
Let $n > 1$ be an integer.
We put $H=\{h\in\N^d \mid \sum_{i=1}^d\sigma_i(h)\in n\Z\}$ and $E=n\N^d$.
Then $(H, E)$ is a simplicial semigroup of rank $d$.
We then have
\begin{align*}
\Ap(H, E)&=\{h\in H \,|\, \sigma_i(h)<n\text{ for each } 1\le i \le d\}\\
\Soc(H, E)&=\left\{h\in\Ap(H, E) \,\middle|\, 
\sum_{i=1}^d\sigma_i(h)=d \left[\frac{(d-1)n}{d}\right]\right\},
\end{align*}
where the bracket means the Gaussian symbol.
If $H$ is an AG semigroup with $|\Soc(H, E)|>1$,
then $2\deg h=\deg h+n$ for every $h\in\Soc(H, E)$
by Theorem~\ref{simplicialAG},
and
\[
n=\deg h>nd-n-d,
\]
thus $(n-1)(d-2)<2$,
namely either $d\leq2$ or both $d=3$ and $n=2$.
The converse is also true, namely, $d \leq 2$ or both $d=3$ and $n=2$, then $H$ is an AG semigroup.
We note that $|\Soc(H, E)|=1$ if and only if $n$ divides $d$, 
and as is well-known 
that this is also equivalent to the Gorensteinness of $k[H]$.
Since $\fd[H]$ has a standard grading
in this case,
$H$ is an AG semigroup 
if and only if $\fd[H]$ is an \AGG ring
(see Remark~\ref{orth_semi3}),
thus we obtain the same result as \cite[Example 10.8]{GTT}.
\end{ex}

Thanks to Theorem~\ref{simplicialAG},
we can extend \cite[Theorem 4.3]{HJS} as follows.

\begin{cor}\label{normalAG}
Let $(H,E)$ be a normal simplicial semigroup. 
Assume that $\omega_H = \relint H$ is 
minimally generated by $a_1, \dots, a_s$, where $s > 0$. 
Then $H$ is an AG semigroup 
if and only if 
there exists $1 \le i \le s$ 
such that, for any $j\ne i, 1 \le j \le s$, 
there are 
$1 \le j'\le s$ and $b \in E$
satisfying $a_j + a_{j'} = a_i + \sum_{h \in E \setminus \{b\}} h$. 
Under these conditions, $a_i$ is an Ulrich element of $H$.

In particular, if $\rank H =2$, $a_i$ is an Ulrich element 
if and only if for every $j \ne i, 1\le j \le s$,
one can find $1 \le j' \le s$ and $b \in E$ 
such that $a_j + a_{j'} = a_i + b$.
\end{cor}

\begin{proof}
Since $H$ is normal,
for each $1 \le i\le s$,
there exists $h\in\Soc(H, E)$
such that $a_i=\sum_{b\in E}b-h$.
From this, the assertion follows from Theorem~\ref{simplicialAG}.
\end{proof}

\subsection{Examples}
Let us show some concrete examples.

\begin{ex}
Let
\[
H=\begin{pmatrix} 5 & 0 & 0 & 2 & 1  \\ 0 & 5 & 0 & 1 & 3 \\
0 & 0 & 5 & 1 & 3\end{pmatrix}
\quad\text{and}\quad
E=\begin{pmatrix} 5 & 0 & 0 \\ 0 & 5 & 0 \\ 0 & 0 & 5 \end{pmatrix}.
\]
Then $(H, E)$ is a normal simplicial semigroup
and
\[
2\begin{pmatrix}1 \\ 3 \\ 3\end{pmatrix}=
\begin{pmatrix}2 \\ 1 \\ 1\end{pmatrix}
+\begin{pmatrix}0 \\ 5 \\ 0\end{pmatrix}
+\begin{pmatrix}0 \\ 0 \\ 5\end{pmatrix}
\]
implies that $H$ is an AG semigroup.
Furthermore, $\fd[H]$ is an \AGG ring,
since $H$ is slim.
\end{ex}

\begin{thm}\label{cyc_invariant}
Let $n>m_1$ be coprime positive integers and define $H \subset \N^2$ by  
\[ 
H = \left\{\binom{u_1}{u_2} 
~\middle|~  \;  u_1 + m_1u_2 \equiv 0  \pmod n\right\}.
\]

We choose $m_2<n$ with $m_1m_2\equiv1\pmod n$ and put 
\[c=\dfrac{m_1m_2-1}n\in\N.\]

Then the following conditions are equivalent.
\begin{enumerate}
\item $H$ is an AG semigroup.
\item  $m_1\equiv m_2\equiv1\pmod c$.
\item The  Hirzebruch-Jung continued fraction expansion 
of $\frac n{m_1}$ (for the definition, see \cite{NY}) is
\[
\dfrac n{m_1} = [[q+1, \underbrace{2, \dots, 2}_{c-1}, p+1]],
\]
where $2$ appears continuously $c-1$ times.
\end{enumerate}

If, moreover, $m_1 >1$ and conditions (1) - (3) are satisfied, then $\relint(H)$ 
has unique element $v$ with $\deg(v) < \deg (h)$ for any $h \in H, h\ne 0$, 
which is an Ulrich element of $H$.

Note that if $k$ has  a primitive $n$-th root of unity $\xi$,
then $\fd[H]=\fd[t_1, t_2]^G$ is the invariant subring of a cyclic group 
$G\subset GL(2,k)$ of order $n$     
generated by 
$\begin{pmatrix}\xi & 0 \\ 0 & \xi^{m_1}\end{pmatrix}$. 
\end{thm}

\begin{proof} 
It is obvious that $(H, E)$ is orthogonal of order $n$, where 
$E=\{n\e_1, n\e_2\}$. 

Let $B_H\subset\omega_H=\relint H$
be the set of minimal generating system of $\omega_H$.
Note that $H$ is normal (cf. Proposition~\ref{normal})
and that $B_H$
is contained in
\[
\Ap(H, E)
=\left\{\binom{u_1}{u_2}\in H  ~\middle|~ 0\leq u_1<n, 0\leq u_2<n\right\}\]
with $|\Ap(H, E)|=n$. The following Claim is the key of our proof. 

\bigskip

{\bf Claim.}  $v=\binom pq$ is an Ulrich element
if and only if 
\[
B_H=\left\{\binom{u_1}{u_2}\in\Ap(H, E) ~\middle|~
0<u_1\leq p\text{ or }0<u_2\leq q
\right\}
\tag{$*$}
\]
and $|B_H|=p+q-1$.

To prove the Claim, first note that, since $H$ is normal, 
$\binom{u_1}{u_2}\notin B_H$ if $u_1\geq p$ and $u_2\geq q$,  
since $\binom{u_1}{u_2}-v\in H$.
 
Assume that $v$ is an Ulrich element
and $p>1$.
For $i=1, \dots, p-1$,
we denote the element in $\Ap(H, E)$
whose first entry is $i$
by $a_i$.
Suppose that
there exists $i$ with $a_i\notin B_H$.
Then $i>1$, since $a_1\in B_H$, and we may assume
$a_{i'}\in B_H$ for $i'=1, \dots, i-1$.
Since $a_i\notin B_H$,
there is $w\in\Ap(H, E)$ with $w\leq_Ha_i$.
From $\sigma_1(w)<\sigma_1(a_i)=i$,
it follows that
there is $i'<i$ with $w=a_{i'}$. 
Since $v$ is an Ulrich element, there exists $u \in B_H$ 
with  $a_{i'}+u =v+b$ for some $u\in E$ by Corollary~\ref{normalAG}.
Looking at the first component, we see that $u = a_{p-i'}$ and 
 $a_{i'}+a_{p-i'}=v+n\e_2$. It is easy to see that we have also 
 $a_i + a_{p-i} = v + \e_2 = a_{i'}+a_{p-i'}$. 
Then
\[
a_{p-i'}=v+n\e_2-a_{i'}=a_i+a_{p-i}-a_{i'}
=(a_i-a_{i'})+a_{p-i}
\]
implies $a_{p-i'}\notin B_H$, 
contradicting our assertion 
$a_{p-i'}\in B_H$.
Hence such $i$ does not exist
and $a_i\in B_H$ for $i=1, \dots, p-1$.
Similarly,
$\binom{u_1}{u_2}\in\Ap(H, E)$ is contained in $B_H$,
if $0<u_2<q$.
Hence $(*)$ holds.
Conversely, if $(*)$ holds,
$v$ is an Ulrich element by Corollary~\ref{normalAG}
and the claim is proved.

Assume that $v=\binom pq$ is an Ulrich element.
Then 
$a_1=\binom1{n-m_2}$ and
$q=\sigma_2(v)<\sigma_2(a_{p-1})<\cdots<\sigma_2(a_1)=n-m_2$.
Since $\sigma_2(a_i)\equiv i\sigma_2(a_1)\pmod n$,
we have $\sigma_2(a_i)=n-im_2$ for $i=1, \dots, p-1$ 
and $q=\sigma_2(v)=n-pm_2$.
Similarly, we have $p=n-qm_1$,
thus $p(m_2-1)=q(m_1-1)$.
Hence $m_1-1=p'c'$ and $m_2-1=q'c'$
where $p'=p/g$, $q'=q/g$,
$c'\in\N$ and $g=\gcd(p, q)$.
And $n=p'qc'+p+q$
Note that $g$ divides $n$
and 
\[
nc'/g=(p'c'+1)(q'c'+1)-1=m_1m_2-1.
\]
Since $m_1m_2-1=nc$,
we have $c'=cg$, thus
\[
m_1=p'c'+1=pc+1
\qquad\text{and}\qquad
m_2=qc+1.
\]
Hence $m_1\equiv m_2\equiv1\pmod c$.
Conversely, assume
$m_1=pc+1$ and $m_2=qc+1$
where $p, q\in\N$.
Then $n=pqc+p+q$
and $(*)$ holds,
thus $v=\binom{p}{q}$ is an Ulrich element.
In this case, the Hirzebruch-Jung continued fraction expansion 
of $\frac n{m_1}$ is directly computed as described.

Finally, if $m_1>1$,
then $\sigma_1(u)+\sigma_2(u)<n$ for $u\in\Soc(H)$.
If $v$ is an Ulrich element
and $u\in\Soc(H)\setminus\{v\}$,
then there are $w\in\Soc(H)$ and $n\e_i\in E$
with $u+w=v+n\e_i$.
Then
\[
\sigma_1(v)+\sigma_2(v)
=\sigma_1(u+w)+\sigma_2(u+w)-n<\sigma_1(u)+\sigma_2(u),
\]
which implies the uniqueness of $v$.
This completes the proof.
\end{proof}

\begin{ex}
In Theorem~\ref{cyc_invariant},
if $n=7$,
then
\[
6\cdot6\equiv5\cdot3\equiv4\cdot2\equiv1\pmod7.
\]
If we take $m_1=6$ (resp. $5$ or $4$),
then $c=5$ (resp. $2$ or $1$)
and hence $m_1\equiv m_2\equiv1\pmod c$ in any case.
Therefore, $H$ has an Ulrich element.

If we choose $n=11$,
then
\[
10\cdot10\equiv9\cdot5\equiv8\cdot7\equiv6\cdot2\equiv4\cdot3
\equiv1\pmod{11}.
\]
If $m_1=10$ (resp. $9, 8, 6$ or $4$),
then $c=9$ (resp. $4, 5, 1$ or $1$).
Hence, $H$ has an Ulrich element unless $m_1=8$ or $7$.
\end{ex}

Next, we observe the case in which semigroups are not normal.

\begin{rem}\label{rem_proj}
Let $(H, E)$ be an orthogonal semigroup of order $m$ and rank $d$.
Put $H_i=\{ \sigma_i(h) \,|\, h\in H\}$ for each $1 \le i \le d$.
If $H$ is an AG semigroup of type two,
then $H_i$ is an AG semigroup of type at most two for every $1 \le i \le d$.
\end{rem}

\begin{proof}
Suppose that $H$ is an AG semigroup of type two 
and write $\Soc(H,E) = \{u,v\}$.
Then, by Theorem~\ref{simplicialAG}, 
we may assume $v + m \e_i = 2u$ for some $1 \le i \le d$ 
and $w \leq_H v$ for every $w \in \Ap(H,E) \setminus\{u\}$. 
We also may assume $i=1$ by switching the rows of $H$ suitably.
Then we have $v + m \e_1 = 2u$.
For each $c \in \Ap(H_i, \{m\})$, 
there exists $h \in \Ap(H,E)$ such that $\sigma_i(H) = c$ by Theorem~\ref{symp_CM}.
Moreover, 
since $|\Ap(H, E)|=|\Ap(H_i, \{m\})|=m$ 
(\cite[Theorem 4.6]{CN}),
the correspondence $c \mapsto h$ is one-to-one.
It follows that
$H_i$ is a symmetric numerical semigroup
with $\Soc(H_i, \{m\})=\{\sigma_i(v)\}$ for every $1 < i \le d$, i.e., $\type H_i = 1$.
Suppose that $H_1$ is not $\langle1\rangle = \N$.
Then $\sigma_1(v)>\sigma_1(u) > m$ implies that $\sigma_1(v)$ is maximal in $\Ap(H_1, \{m\})$.
Since $\sigma_1(u)\not\leq_H\sigma_1(v)$,
$H_1$ is a pseudo-symmetric numerical semigroup, which is equivalent to being an AG semigroup of type two.
\end{proof}

We consider some examples to illustrate the meaning of Remark \ref{rem_proj}.
In the following example, both $H$ define projective monomial curves.

\begin{ex}
Let $H=\begin{pmatrix}0&4&5&7\\ 7&3&2&0\end{pmatrix}$.
Then $\Soc(H, E)=\left\{\binom{13}8, \binom{16}4\right\}$ 
and 
$2\binom{10}4=\binom{13}8+\binom70$ implies
that $H$ is an AG semigroup.
In this example, $H_1 = \langle 4,5,7\rangle$ is pseudo-symmetric and $H_2 = \langle 2,3,7\rangle = \langle 2,3\rangle$ is symmetric.\par
On the other hand, it may happen that every $H_i$ is of type one.
For example,
let $H=\begin{pmatrix}0&1&3&7\\ 7&6&4&0\end{pmatrix}$
and $E=\begin{pmatrix}0&7\\ 7&0\end{pmatrix}$.
Then $\Soc(H, E)=\left\{\binom68, \binom5{16}\right\}$
and $2\binom68=\binom5{16}+\binom70$,
thus, $H$ is an AG semigroup.
Notice that both $H_1 = \langle 1,3,7\rangle = \langle 1 \rangle$ and $H_2 = \langle 4,6,7\rangle$ are symmetric, namely, of type one.
\end{ex}

The next example shows that $H_i$ needs not to be an AG semigroup, even if $H$ is an AG semigroup, unless the type of $H$ is at most two.

\begin{ex}
Let
\[
E=\begin{pmatrix} 8 & 0 \\ 0 & 8 \end{pmatrix}
\text{ and }
H=\begin{pmatrix}
8 & 0 & 9 & 22 & 31  \\
0 & 8 & 7 & 18 & 17 
\end{pmatrix}.
\]
Then
\[
\Ap(H, E)
=\left\{
\binom{9a}{7a} \,\middle|\, a=0, \dots,  5\right\}
\cup\left\{
\binom{22}{18}, \binom{31}{17}\right\},
\]
\[
\Soc(H, E) = \left\{\binom{22}{18}, \binom{31}{17}, \binom{45}{35}\right\},
\]
and $\fd[H]$ is Cohen-Macaulay.
Since
\[
\binom{22}{18}+\binom{31}{17}
=\binom{45}{35}+\binom80,
\]
$H$ is an AG semigroup of type three.
On the other hand,
$H_1=\langle8, 9, 22\rangle$ is not an AG semigroup
since $\Soc(H,\{8\}) = \{31,45\}$ and $31+31\ne45+8$
(note that $H_2=\langle7, 8, 17, 18\rangle$ is symmetric).
\end{ex}

\begin{ex}\label{notmonomial}
Let $H=\begin{pmatrix} 0 & 1 & 3 & 8 \\ 8 & 3 & 1 & 0\end{pmatrix}$.
Then
\[
\Soc(H, E)=\left\{
\binom57, \binom75
\right\},
\]
$2\binom57-\binom75 = \binom{3}{9}\notin E$, and
$2\binom75-\binom57  = \binom{9}{3}\notin E$
imply that $H$ is not an AG semigroup.
However, $R=\fd[H]$ is an \AGG ring.
First, notice that $\omega_H=\relint H$ since $H$ is normal
Let $f=t^{\binom13}-t^{\binom31}\in\K_R$. 
Then, $f$ is a homogeneous element of $R$ of degree $4$.
Moreover, we can easy to see that
 \[
\dim_{\fd}(\K_R/fR)_n=
\begin{cases}1 & \text{if $n$ is a multiple of $4$},\\
0 & \text{otherwise}.
\end{cases}
\]
for all $n \in \Z$. Since $(t^{\binom13})$ is a reduction of $\mm$ with respect to $\K_R/fR$,
\[
\mult(\K_R/fR)
=\lim_{t\to1}(1-t^4)(t^4+t^8+\cdots)=1=\type H-1.
\]
Therefore, $R$ is an \AGG ring.
This example certifies that
$H$ is not always an AG semigroup,
even if $\fd[H]$ is an \AGG ring.
\end{ex}

\end{document}